\documentclass[12pt]{article}
\usepackage{amsmath, amsthm, amssymb}

\theoremstyle{plain}
\newtheorem{thm}{Theorem}

\newtheorem{lem}[thm]{Lemma}
\newtheorem{cor}[thm]{Corollary}
\newtheorem*{MainLemma}{Main Lemma}
\newtheorem*{LovaszLemma}{Lovasz's Decomposition Lemma}

\theoremstyle{definition}
\newtheorem{defn}{Definition}

\theoremstyle{remark}

\newtheorem*{question}{Question}

\title{Destroying Non-Complete Regular Components in Graph Partitions}
\author{Landon Rabern\\
\small \texttt{landon.rabern@gmail.com}}
\begin{document}
\maketitle

\begin{abstract}
We prove that if $G$ is a graph and $r_1, \ldots, r_k \in \mathbb{Z}_{\geq 0}$ such that $\sum_{i=1}^k r_i \geq \Delta(G) + 2 - k$ then $V(G)$ can be partitioned into sets $V_1, \ldots, V_k$ such that $\Delta(G[V_i]) \leq r_i$ and $G[V_i]$ contains no non-complete $r_i$-regular components for each $1 \leq i \leq k$.  In particular, the vertex set of any graph $G$ can be partitioned into $\left \lceil \frac{\Delta(G) + 2}{3} \right \rceil$ sets, each of which induces a disjoint union of triangles and paths.
\end{abstract}

\section{Introduction}
In \cite{Kostochka} Kostochka modified an algorithm of Catlin to show that every triangle-free graph $G$ can be colored with at most $\frac{2}{3} \left (\Delta(G) + 3 \right)$ colors.  In fact, his modification proves that the vertex set of any triangle-free graph $G$ can be partitioned into $\left \lceil \frac{\Delta(G) + 2}{3} \right \rceil$ sets, each of which induces a disjoint union of paths. We generalize this as follows.

\begin{MainLemma}
Let $G$ be a graph and $r_1, \ldots, r_k \in \mathbb{Z}_{\geq 0}$ such that $\sum_{i=1}^k r_i \geq \Delta(G) + 2 - k$. Then $V(G)$ can be partitioned into sets $V_1, \ldots, V_k$ such that $\Delta(G[V_i]) \leq r_i$ and $G[V_i]$ contains no non-complete $r_i$-regular components for each $1 \leq i \leq k$.
\end{MainLemma}

Setting $k = \left \lceil \frac{\Delta(G) + 2}{3} \right \rceil$ and $r_i = 2$ for each $i$ gives a slightly more general form of Kostochka's theorem.

\begin{cor}
The vertex set of any graph $G$ can be partitioned into $\left \lceil \frac{\Delta(G) + 2}{3} \right \rceil$ sets, each of which induces a disjoint union of triangles and paths.
\end{cor}

For coloring, this actually gives the bound $\chi(G) \leq 2  \left \lceil \frac{\Delta(G) + 2}{3} \right \rceil$ for triangle free graphs.  To get $\frac{2}{3} \left (\Delta(G) + 3 \right)$, just use $r_k = 0$ when $\Delta \equiv 2 (\text{mod } 3)$.\newline

Similarly, for any $r \geq 2$, setting $k = \left \lceil \frac{\Delta(G) + 2}{r + 1} \right \rceil$ and $r_i = r$ for each $i$ gives the following.
\begin{cor}
Fix $r \geq 2$.  The vertex set of any $K_{r + 1}$-free graph $G$ can be partitioned into $\left \lceil \frac{\Delta(G) + 2}{r + 1} \right \rceil$ sets each inducing an $(r-1)$-degenerate subgraph with maximum degree at most $r$.
\end{cor}

For the purposes of coloring it is more economical to split off $\Delta + 2 - (r+1)\left \lfloor \frac{\Delta + 2}{r + 1} \right \rfloor$ parts with $r_j = 0$.

\begin{cor}
Fix $r \geq 2$.  The vertex set of any $K_{r + 1}$-free graph $G$ can be partitioned into $\left \lfloor \frac{\Delta(G) + 2}{r + 1} \right \rfloor$ sets each inducing an $(r-1)$-degenerate subgraph with maximum degree at most $r$ and $\Delta(G) + 2 - (r+1)\left \lfloor \frac{\Delta(G) + 2}{r + 1} \right \rfloor$ independent sets.  In particular, $\chi(G) \leq \Delta(G) + 2 - \left \lfloor \frac{\Delta(G) + 2}{r + 1} \right \rfloor$.
\end{cor}

For $r \geq 3$, the bound on the chromatic number is only interesting in that its proof does not rely on Brooks' theorem.  In \cite{Lovasz} Lov\'{a}sz proved a decomposition lemma of the same form as the Main Lemma.  The Main Lemma gives a more restrictive partition at the cost of replacing $\Delta(G) + 1$ with $\Delta(G) + 2$.

\begin{LovaszLemma}
Let $G$ be a graph and $r_1, \ldots, r_k \in \mathbb{Z}_{\geq 0}$ such that $\sum_{i=1}^k r_i \geq \Delta(G) + 1 - k$. Then $V(G)$ can be partitioned into sets $V_1, \ldots, V_k$ such that $\Delta(G[V_i]) \leq r_i$ for each $1 \leq i \leq k$.
\end{LovaszLemma}

For $r \geq 3$, combining this with Brooks' theorem gives the following better bound for a $K_{r + 1}$-free graph $G$ (first proved in \cite{BK}, \cite{CatlinFirst} and \cite{Lawrence}):

\[\chi(G) \leq \Delta(G) + 1 - \left \lfloor \frac{\Delta(G) + 1}{r + 1} \right \rfloor.\]

\section{The Proofs}
Instead of proving directly that we can destroy all non-complete $r$-regular components in the partition, we prove the theorem for the more general class of $r$-permissible graphs and show that non-complete $r$-regular graphs are $r$-permissible.

\begin{defn}
For a graph $G$ and $r \geq 0$, let $G^r$ be the subgraph of $G$ induced on the vertices of degree $r$ in $G$.
\end{defn}

\begin{defn}
Fix $r \geq 2$.  A collection $T$ of graphs is $r$-\emph{permissible} if it satisfies all of the following conditions.
\begin{enumerate}
\item Every $G \in T$ is connected.
\item $\Delta(G) = r$ for each $G \in T$.
\item $\delta(G^r) > 0$ for each $G \in T$.
\item If $G \in T$ and $x \in V(G^r)$, then $G - x \not \in T$.
\item If $G \in T$ and $x \in V(G^r)$, then there exists $y \in V(G^r) - N_G(x)$ such that $G - y$ is connected.
\item Let $G \in T$ and $x \in V(G^r)$. Put $H = G - x$.  Let $A \subseteq V(H)$ with $|A| = r$.  Let $y$ be some new vertex and form $H_A$ by joining $y$ to $A$.  If $H_A \in T$, then $A \cap N_G(x) \cap V(G^r) \neq \emptyset$.
\end{enumerate}

For $r = 0, 1$ the empty set is the only $r$-permissible collection.
\end{defn}

\begin{lem}
Fix $r \geq 2$ and let $T$ be the collection of all non-complete connected $r$-regular graphs.  Then $T$ is $r$-permissible.
\end{lem}
\begin{proof}
Note that for $G \in T$ we have $G^r = G$.  Now (1), (2), (3) and (4) are clearly satisfied.  Since each $G \in T$ is non-complete, it has at least two end blocks.  Thus if $x \in V(G^r)$ we can always pick some $y$ such that $G-y$ is connected in an end block not containing $x$.  Hence (5) holds.  That (6) holds is immediate from regularity.  Hence $T$ is $r$-permissible.
\end{proof}

Now to prove the Main Lemma we just need to prove the following.

\begin{lem}
Let $G$ be a graph and $r_1, \ldots, r_k \in \mathbb{Z}_{\geq 0}$ such that $\sum_{i=1}^k r_i \geq \Delta(G) + 2 - k$. If $T_i$ is an $r_i$-permissible collection for each $1 \leq i \leq k$, then $V(G)$ can be partitioned into sets $V_1, \ldots, V_k$ such that $\Delta(G[V_i]) \leq r_i$ and $G[V_i]$ contains no element of $T_i$ as a component for each $1 \leq i \leq k$.
\end{lem}
\begin{proof}
For a graph $H$, let $c(H)$ be the number of components in $H$ and let $p_i(H)$ be the number of components of $H$ that are members of $T_i$.  For a partition $P = (V_1, \ldots, V_k)$ of $V(G)$ let

\[f(P) = \sum_{i=1}^k \left(|E(G[V_i])| - r_i|V_i|\right),\]
\[c(P) = \sum_{i=1}^k c(G[V_i]),\]
\[p(P) = \sum_{i=1}^k p_i(G[V_i]).\]

Let $P = (V_1, \ldots, V_k)$ be a partition of $V(G)$ minimizing $f(P)$ and then $c(P)$ and then $p(P)$.\newline

Let $1 \leq i \leq k$ and $x \in V_i$ with $d_{G[V_i]}(x) \geq r_i$.  Since $\sum_{i=1}^k r_i \geq \Delta(G) + 1 - k$ there is some $j \neq i$ such that $d_{G[V_j]}(x) \leq r_j$.  Moving $x$ from $V_i$ to $V_j$ gives a new partition $P^*$ with $f(P^*) \leq f(P)$.  Note that if $d_{G[V_i]}(x) > r_i$ we would have $f(P^*) < f(P)$ contradicting the minimality of $P$. This proves that $\Delta(G[V_i]) \leq r_i$ for each $1 \leq i \leq k$.\newline

Now assume that for some $i_1$ there is $A_1 \in T_{i_1}$ which is a component of $G[V_{i_1}]$.  Plainly, we may assume that $r_{i_1} \geq 2$.  Put $P_1 = P$ and $V_{1,i} = V_i$ for $1 \leq i \leq k$. Take $x_1 \in V(A_1^{r_{i_1}})$ such that $A_1 - x_1$ is connected (this exists by condition (5) of $r$-permissibility).  By the above we have $i_2 \neq i_1$ such that moving $x_1$ from $V_{1, i_1}$ to $V_{1, i_2}$ gives a new partition $P_2 = (V_{2, 1}, V_{2,2}, \ldots, V_{2,k})$ such that $f(P_2) = f(P_1)$.  By the minimality of $c(P_1)$, $x_1$ is adjacent to only one component $C_2$ in $G[V_{2, i_2}]$. Let $A_2 = G[V(C_2) \cup \{x_1\}]$.  Since (by condition (4)) we destroyed a $T_{i_2}$ component when we moved $x_1$ out of $V_{1, i_1}$, by the minimality if $p(P_1)$, it must be that $A_2 \in T_{i_2}$. Now pick $x_2 \in A_2^{r_{i_2}}$ not adjacent to $x_1$ such that $A_2 - x_2$ is connected (again by condition (5)).  Continue on this way to construct sequences $i_1, i_2, \ldots$, $A_1, A_2, \ldots$, $P_1, P_2, P_3, \ldots$ and $x_1, x_2, \ldots$.  Since $G$ is finite, there is a smallest $t$ such that $A_{t + 1} - x_t = A_s - x_s$ for some $s < t$.  Put $B = A_s - x_s$. By condition (6) of $r_{i_s}$-permissibility, we have $z \in N_{B}(x_t) \cap N_{B}(x_s) \cap A_s^r$.\newline

We now modify $P_s$ to contradict the minimality of $f(P)$.  Consider the set $X = \{x_{s+1}, x_{s+2}, \ldots, x_{t-1}\} \cap V_{s, i_s}$.  For $x_j \in X$, since $x_j \in A_j^r$ and $x_j$ is not adjacent to $x_{j-1}$ we see that $d_{G[V_{s, i_s}]}(x_j) \geq r_{i_s}$.  Similarly, $d_{G[V_{s, i_t}]}(x_t) \geq r_{i_t}$. Also, by the minimality of $t$, $X$ induces an independent set in $G$.  Thus we may move all elements of $X$ out of $V_{s, i_s}$ to get a new partition $P^* = (V_{*, 1}, \ldots, V_{*, k})$ with $f(P^*) = f(P)$.  Since $x_t$ is adjacent to exactly $r_{i_s}$ vertices in $V_{t+1, i_s}$ and the only possible neighbors of $x_t$ that were moved out of $V_{s, i_s}$ between steps $s$ and $t+1$ are the elements of $X$, we see that $d_G[V_{*, i_s}](x_t) = r_{i_s}$.  Since $d_{G[V_{*, i_t}]}(x_t) \geq r_{i_t}$ we can move $x_t$ from $V_{*, i_t}$ to $V_{*, i_s}$ to get a new partition $P^{**} = (V_{**, 1}, \ldots, V_{**, k})$ with $f(P^{**}) = f(P^*)$.  Now, remember our vertex $z \in V_{**, i_s}$.  Since $z$ is adjacent to $x_t$ we have $d_{G[V_{**, i_s}]}(z) \geq r_{i_s} + 1$.  Thus we may move $z$ out of $V_{**, i_s}$ to get a new partition $P^{***}$ with $f(P^{***}) < f(P^{**}) = f(P)$.  This contradicts the minimality of $f(P)$.
\end{proof}

\begin{question}
Are there any other interesting $r$-permissible collections?
\end{question}

\begin{question}
The definition of $r$-permissibility can be weakened in various ways and the proof will still go through.  Does this yield anything interesting?
\end{question}

\section*{Acknowledgments}
Thanks to Dieter Gernert for finding and sending me a copy of Kostochka's paper.

\end{document}